\documentclass[leqno,12pt]{amsart} 

\usepackage{amsmath,amssymb,verbatim}

\usepackage[pagebackref,pdftex,hyperindex]{hyperref}

\usepackage{color}
\usepackage{braket}
\usepackage{bm}
\usepackage{url}
\usepackage{comment}

\theoremstyle{plain} 
\newtheorem{theorem}{\noindent\bf Theorem}[section] 
\newtheorem{lemma}[theorem]{\noindent\bf Lemma}

\theoremstyle{definition}

\newcommand{\C}{\mathbb{C}}

\newcommand{\R}{\mathbb{R}}
\renewcommand{\S}{\mathbb{S}}

\newcommand{\ft}{\mathfrak{t}}

\newcommand{\cE}{\mathcal{E}}

\newcommand{\cH}{\mathcal{H}}

\newcommand{\cT}{\mathcal{T}}

\renewcommand{\a}{\alpha}

\renewcommand{\d}{\delta}
\newcommand{\e}{\varepsilon}
\newcommand{\s}{\sigma}

\renewcommand{\phi}{\varphi}

\newcommand{\eg}{{\rm e.g.\ }}

\renewcommand{\leq}{\leqslant}
\renewcommand{\geq}{\geqslant}
\newcommand{\abs}[1]{\left\lvert#1\right\rvert}
\newcommand{\norm}[1]{\left\|#1\right\|} 
\newcommand{\ip}[1]{\left\langle#1\right\rangle}

\renewcommand{\Re}{\mathrm{Re}}

\DeclareMathOperator{\Aut}{Aut}

\DeclareMathOperator{\dd}{\sqrt{-1}\partial\bar{\partial}}
\DeclareMathOperator{\dbar}{\overline{\partial}}

\DeclareMathOperator{\Ent}{Ent}

\DeclareMathOperator{\Fut}{Fut}

\DeclareMathOperator{\id}{id}

\DeclareMathOperator{\MA}{MA}

\DeclareMathOperator{\Ric}{Ric}

\DeclareMathOperator{\Tr}{Tr}

\numberwithin{equation}{section}       

\theoremstyle{plain}
\newtheorem{prop} {Proposition} [section]
\newtheorem{thm}[prop] {Theorem} 
\newtheorem{lem}[prop] {Lemma}

\theoremstyle{definition}
\newtheorem{dfn}[prop] {Definition}

\newtheorem{rem}[prop]{Remark}
 
\newtheorem*{ackn}{\bf{Acknowledgment}}

\begin{document}

\title[Continuity method for the Mabuchi soliton]
{Continuity method for the Mabuchi soliton
on the extremal Fano manifolds} 
\author[T. Hisamoto and S. Nakamura]
{Tomoyuki Hisamoto and Satoshi Nakamura} 

\thanks{ 
}
\address{
T. Hisamoto:
Department of Mathematical Sciences, Tokyo Metropolitan University, 1-1 Minamiosawa, Hachioji-shi, Tokyo 192-0397, Japan
}
\email{hisamoto@tmu.ac.jp}
\address{
S. Nakamura:
Department of Mathematics, Institute of Science Tokyo,
2-12-1 Ookayama, Meguro-ku, Tokyo 152-8551, Japan
}
\email{s.nakamura@math.titech.ac.jp}

\begin{abstract}
We run the continuity method 
for 
Mabuchi's generalization of K\"ahler-Einstein metrics, 
assuming the existence of an extremal K\"ahler metric. 
It gives an analytic proof 
(without minimal model program) 
of the recent existence 
result obtained by Apostolov, Lahdili and Nitta. 
Our key observation is 
the boundedness of the energy functionals 
along the continuity method. 
The same argument can be applied to general $g$-solitons 
and $g$-extremal metrics. 
\end{abstract}

\maketitle

\tableofcontents

\section{Introduction}
The celebrated work of Chen-Donaldson-Sun 
\cite{CDS1} and Tian \cite{Tia15} states that
a Fano manifold admits a K\"ahler-Einstein metric if and only if it is K-polystable. 
There are several generalization of K\"ahler-Einstein metrics, 
which may exist even if the Fano manifold is obstructed by Futaki's invariant \cite{Fu83}. 
One is the extremal K\"ahler metric 
introduced by Calabi \cite{Ca85}. 
It is defined in terms of the scalar curvature so that 
make sense for general polarized manifolds. 
In \cite{Ma01} Mabuchi introduced an analogous notion,
called the Mabuchi soliton,
in terms of the Ricci potential function. 
If the Futaki invariant vanishes,
it is classically known that 
these two notions coincide with the K\"ahler-Einstein metric. 
As these are different metrics in general,
one may only ask if the existence conditions 
are equivalent.  
This is the main issue discussed in this article.

The one direction is relatively clear; 
if the Fano manifold admits a Mabuchi soliton, 
there exists an extremal metric,
as the second author pointed out in \cite[Section 9.6]{Mabook} for example.
See also \cite{Hi19-2, Ya22}. 
The recent result \cite{ALN24} 
states that the converse is also true 
if one assumes the condition on the Mabuchi constant (see Theorem \ref{main cor}). 
The proof in \cite{ALN24} is intricated 
as it exploits the deep result of the minimal model program 
in algebraic geometry. 

The aim of the present paper 
is to provide a direct analytic proof 
for the existence of the Mabuchi soliton. 
The strategy is completely different from \cite{ALN24}
and includes new arguments for the continuity method
for the Mabuchi soliton
which are a boundedness argument (Theorem \ref{Thm-M})
and a related compactness argument (Section \ref{pr-closed}).
These new argument are inevitable
since the monotonicity property of the energy 
along such path might not be 
expected as discussed after Theorem \ref{Thm-M}.

Our argument can be applied to general 
 {\em $g$-soliton metric} in the sense of 
\cite{BW14}, \cite{HaLi20}. 
In order to describe the general situation, 
let $X$ be an $n$-dimensional Fano manifold and 
fix a maximal compact torus $T \subset \Aut(X)$.
Since the $T$-action lifts to the anti-canonical bundle,
we have the the moment map 
$\mu_\omega \colon X \to \ft^*$
for any $T$-invariant K\"ahler metric 
$\omega\in 2\pi c_{1}(X)$.
Let $P$ be the associated weight polytope as the image of $\mu_{\omega}$.
We take a smooth positive function $g \colon P \to \R$ 
and denote the composition by $g_\omega :=g \circ \mu_\omega$. 
We also assume that $g$ is log concave\footnote
{The log concavity assumption is only used in the regularity argument,
which is based on \cite{HaLi20}, Proposition $3.8$,
for a solution of the continuity method \eqref{conti-eq}
in section \ref{sol-at-0} and \ref{pr-closed}.
One can also apply the argument in \cite{DJL24, HaLiu24}.
We are not sure if the assumption can be relaxed.}.

\begin{dfn}
    A $T$-invariant K\"ahler metric $\omega\in 2\pi c_{1}(X)$
    is called {\em a $g$-soliton} \cite{HaLi20} if it satisfies 
\begin{equation}
    \Ric(\omega)-\omega=\dd\log g_\omega. 
\end{equation}
    It is called {\em a $g$-extremal} (c.f. \cite{NN24}) if
    the scalar curvature satisfies 
\begin{equation}
 S_\omega -n=1-g_\omega.
\end{equation}
\end{dfn} 
When $g$ is an affine function, each $g$-soliton and $g$-extremal metric yields
the Mabuchi soliton and the extremal metric. 
In the case $g=e^{l}$ for some affine function $l$, $g$-soliton yields the K\"ahler-Ricci soliton. 
The $g$-extremal metric is the same notion as
the $(1,n+1-g)$-weighted constant scalar
curvature K\"ahler metric in the sense of \cite{Lah19}.

As it was observed in Mabuchi's original paper \cite{Ma01}, 
the constant 
\begin{equation}
m_X := \sup_X (1-g_\omega) 
\end{equation}
is independent of the metrics 
and $m_X <1$ if there exists a $g$-soliton. 

In order to construct a $g$-soliton, we 
introduce the following variant of the continuity method. 
Let us take a $T$-invariant reference metric $\omega_0\in 2\pi c_{1}(X)$ 
and represent each $T$-invariant metric by a K\"ahler potential $\phi$ 
such that $\omega_{\phi}= \omega_0 +\dd \phi$. 
The continuity path is then described as 
the Monge-Amp\`ere equation: 
\begin{equation}\label{conti-eq}
g_{\omega_{\phi}}\omega_{\phi}^{n}
=e^{-t\phi+\rho_0}\omega_0^n 
\quad\text{for}\quad t\in [0,1].
\end{equation}
The solution at $t=1$ is nothing but a $g$-soliton.
Such a continuity path is originated 
from \cite{Aubin}, 
where the path is formulated for the K\"ahler-Einstein metrics. 
See also \cite{TZ00} for the K\"ahler-Ricci soliton,
\cite{LZ} for the Mabuchi soliton
and \cite{DH21} for more general soliton type metrics.
Our main result is the following. 

\begin{thm}\label{Thm-T}
Let $\mathcal{T}$ be the set of $t\in[0,1]$ such that the equation \eqref{conti-eq}
has a $T$-invariant solution. 
In general $\cT$ is non-empty and open. 
If one assume that $X$ admits a $g$-extremal metric 
and that the Mabuchi constant enjoys $m_X <1$, 
then $\cT=[0, 1]$.
\end{thm}

Recall that the $g$-extremal metric can be characterized as a critical point of the $g$-Mabuchi functional
$M_{g}$ (see (\ref{def of D_g and M_g}) for the definition).  
The following is a key to prove the closedness of $\mathcal{T}$. 

\begin{thm}\label{Thm-M} 
Assume $m_X <1$ and 
fix any $\e\in(0,1)$.
Let $\omega_t$ be a $T$-invariant solution of \eqref{conti-eq} at $t\in(\e,1]$.
There exists a constant $C>0$ independent of $\e$ and $t$ satisfying
$$M_{g}(\omega_t)\leq C\e^{-1}.$$
\end{thm}
We emphasize that the critical point of $M_{g}$ is a $g$-extremal metric
which is different from the $g$-soliton unless $g\equiv 1$. 
When $g\equiv 1$, a $g$-extremal metric and a $g$-soliton coincide with
a K\"ahler-Einstein metric, 
and Bando-Mabuchi \cite{BM87} showed that
$M_{1}(\omega_t)$ is monotonically decreasing in $t$. 
However, for general $g$, we cannot expect the monotonicity of $M_{g}(\omega_t)$.
Our strategy for Theorem \ref{Thm-M} is to give a quantitative bound.
The difficulty of general $g$ case 
is caused by the difference between $M_g$ and
another Mabuchi type functional $F_g$ in the literatures. 
The energy $F_g$ is defined as the
free energy of the Monge-Amp\`ere 
measure with $g$-density: $g_\omega \omega^n$, 
and the critical point is the $g$-soliton.
See Remark \ref{F_g and M_g}.

The $g$-soliton is a critical point of the $g$-Ding functional
(see (\ref{def of D_g and M_g}) for the definition).
Note that, unlike the K\"ahler-Einstein case ($g\equiv 1$ case),
the currently known methods can not quantitatively compare the 
coercivity conditions for the two 
canonical energy functionals $D_g$ and $M_g$, either. 

As a consequence of Theorem \ref{Thm-T}, we have the following.

\begin{thm}\label{main cor}
For any Fano manifold $X$ the followings are equivalent.
\begin{enumerate}
\item[$(1)$] $X$ admits a $g$-soliton. 
\item[$(2)$] It satisfies $m_{X}<1$ and is uniformly $g$-relatively D-stable. 
\item[$(3)$] It satisfies $m_{X}<1$ and admits a $g$-extremal metric in $ c_{1}(X)$.
\end{enumerate}
\end{thm}

Equivalence of $(1)$ and $(2)$ was proved by \cite{HaLi20}.  
Thus Theorem \ref{Thm-T} completes the proof of Theorem \ref{main cor}.
Our result establishes a variant of the Yau-Tian-Donaldson correspondence
for $g$-extremal metrics on Fano manifolds in an analytic way,
which is even new for extremal metrics in the sense of Calabi.
Note that the result  in \cite{ALN24} showed the implication
$(3) \Rightarrow (2)$ in an algebraic way. 

We give some remarks for the case when $g$ is affine
which is one of the most interesting case.
As it was shown by \cite{Ya22}, 
the condition $m_X <1$ follows from the stability. 
Note that the assumption $m_{X}<1$ in $(3)$ is necessary, 
as the second author pointed out 
in \cite[Section 9.5]{Mabook}.
In fact there exists a Fano $3$-fold 
which admits an extremal metric in $2\pi c_1(X)$ 
but $m_X \geq 1$.
By contrast, there exist Fano manifolds admitting
no extremal metric (thus admitting no Mabuchi soliton) in the first Chern class.
For example, one can find Fano 3-folds whose automorphisms
are the additive group $\mathbb{C}^{+}$ in \cite{PCS19}, and
these are obstructed by the result \cite{Le85}, Lemma $1$.
More recently,
a toric Fano $10$-fold admitting no extremal metric 
in the first Chern class was found in \cite{HSY24}.

\begin{ackn}
The authors would like to thank Vestislav Apostolov, Eiji Inoue, Yasufumi Nitta
and Yi Yao for helpful discussion and comments.
The second author expresses his gratitude to Shigetoshi Bando for
encouragements for the problem discussed in this article.
The authors thank the referee for the careful reading and for
numerous useful suggestions which improved the presentation of the article.
The first author is supported by 
JSPS Grant-in-Aid for Scientific
Research (C) No.~21K03229. 
The second author is supported by
JSPS Grant-in-Aid for Science
Research Activity Start-up No.~21K20342
and JSPS Grant-in-Aid for Early-Career Scientists
No.~24K16917.
\end{ackn}

\section{Preliminary}

\subsection{g-solitons and g-extremal metrics}

Throughout this paper 
we consider an  $n$-dimensional Fano manifold $X$. 
We fix a maximal compact torus $T \subset \mathrm{Aut}(X)$ 
and a $T$-invariant K\"ahler metric $\omega_{0}\in 2\pi c_{1}(X)$.
Let
\begin{equation}
\cH(X, \omega_0)^T:=\Set{\phi\in C^{\infty}(X,\mathbb{R}) :
\omega_{\phi}:=\omega_{0}+\dd\phi>0 \text{ and $\phi$ is $T$-invariant}}
\end{equation}
be the set of $T$-invariant K\"ahler potentials.
Here each  $\s \in T_\C$ 
sends $\phi$ to $\sigma[\phi]$ 
which is defined by 
$\sigma^{*}\omega_{\phi}=\omega_{0}+\dd\sigma[\phi]$.
Note that $\sigma[\phi]$ is defined only up to addition of a function,
although this never matters in the following argument.
The Ricci potential is a smooth function 
$\rho_{\phi}$ 
characterized by the property 
\begin{equation}
    \Ric(\omega_{\phi})-\omega_{\phi}=\dd\rho_{\phi}
    \quad\text{and}\quad
    \int_{X}(e^{\rho_{\phi}}-1)\omega_{\phi}^{n}=0. 
\end{equation}
We say that $\omega_{\phi}$ is K\"ahler-Einstein if 
$\rho_{\phi}=0$. 
In terms of the scalar curvature 
$S_\phi := \Tr_{\omega_\phi} \Ric(\omega_\phi)$, 
it is equivalent to say $S_\phi=n$. 

For a convenient let us fix an isomorphism $T\simeq (\S^{1})^{r}$ with some non-negative integer $r$.
For each $\alpha=1,\dots, r$, let $\xi_{\alpha}$ be the holomorphic vector field on $X$
generating the action of the $\alpha$-th factor of $T_{\C}\simeq (\mathbb{C}^{*})^{r}$.
Once the metric $\phi\in\cH(X, \omega_0)^T$ is fixed, 
each vector $\xi \in \ft_\C$ defines the Hamilton function
$\theta_\xi(\phi)\in C^{\infty}(X; \R)$
which is characterized by the properties 
such that 
\begin{equation}
i_\xi\omega_{\phi}=\sqrt{-1}\dbar\theta_\xi(\phi)\quad\text{and}\quad
\int_{X}\theta_\xi (\phi)\omega_{\phi}^{n}=0. 
\end{equation}
We set $\theta_\a := \theta_{\xi_\a}$. 
The moment map $\mu_\phi \colon X\to\mathbb{R}^{r}$ is 
then described as 
$\mu_{\phi}(x)=(\theta_{1}(x),\dots,\theta_{r}(x)).$
It is well-known that the image $P:=\mu_{\phi}(X)$
defines a convex polytope in $\mathbb{R}^{r}$ and independent of the choice of 
$\phi\in\cH(X, \omega_0)^T$.

We will also fix a smooth positive function $g\colon P_{X}\to\mathbb{R}$
and put $g_{\phi}:=g\circ \mu_{\phi}$.
We also assume that $g$ is log concave.
In what follows it is always normalized such that 
\begin{equation}
\int_{X}(g_\phi-1)\omega_{\phi}^n =0 
\end{equation}
holds. 
Thanks to the compactness 
of $X$ there exists a uniform constant $C>0$
independent of the choice of $\phi\in\cH(X, \omega_0)^T$
such that
\begin{equation}\label{bound of g}
    C^{-1}\leq g_{\phi} \leq C 
\end{equation}

\begin{dfn}
    A $T$-invariant K\"ahler metric $\omega\in2\pi c_{1}(X)$
    is called a $g$-soliton \cite{HaLi20} if it satisfies 
\begin{equation}
    \Ric(\omega)-\omega=\dd\log g_\omega, 
\end{equation}
or equivalently, $e^{\rho_\phi} =g_\phi$. 
    It is called $g$-extremal (c.f. \cite{NN24}) if
    the scalar curvature satisfies 
\begin{equation}
 S_\omega -n=1-g_\omega.
\end{equation}
\end{dfn} 

The definition of $g$-soliton generalizes the notion of 
the K\"ahler-Einstein metric, the K\"ahler-Ricci soliton,  
and the Mabuchi soliton, as it was explained in the introduction. 
The $g$-extremal K\"ahler metric gives a generalization of 
extremal K\"ahler metrics. 
The notion of general weighted soliton is originated in \cite{Ma03}.
In the work \cite{BW14} 
consideration of general weighted density is motivated by the optimal transport theory. 

Let us consider the 
particular case when $g$ is an affine function. 
In this case $\omega_\phi$ is $g$-soliton if and only if 
$e^{\rho_\phi}-1$ is 
contained in the space of (the above normalized) Hamilton functions 
of $T_\C$. 
We fix a Levi subgroup of $\mathrm{Aut}_{0}(X)$,
or may assume that $\mathrm{Aut}_{0}(X)$ is reductive in our purpose.
The vector field $\eta$ corresponds to $\theta_{\eta}(\phi) = 1-e^{\rho_\phi}$
is algebraically characterized by the property 
\begin{equation}
 \Fut(\xi) +\ip{\xi, \eta}=0, 
\end{equation}
where $\Fut(\xi)$ denotes the 
Futaki invariant \cite{Fu83}. 
Similarly $g$-extremality defines 
a vector field with the property 
$\theta_\eta(\phi) = S(\omega_{\phi})-n$. 
As it was classically known in \cite{FM, Ma01}, 
these two notions are equivalent 
and we call $\eta$ {\em the extremal vector field}. 
One can also check that $\eta$ is contained in the 
center of the Levi subgroup we fixed.
As the first obstruction to the $g$-soliton metric, 
Mabuchi \cite{Ma01} introduced the invariant
\begin{equation}
    m_{X}:=\sup_{x\in X}\theta_\eta (\phi)(x),  
\end{equation}
which is actually independent of the choice of $\phi\in\cH(X, \omega_0)^T$. 
The existence of the Mabuchi soliton forces  
$m_{X}<1$, as the positivity of
$e^{\rho_{\phi}}=1-\theta_{\eta}(\phi).$

\subsection{Energy functionals}
 
Let us briefly review the geometry 
of the space $\cH^T=\cH(X, \omega_0)^T$ and the canonical energy functionals 
there. 
As it was first observed by Mabuchi, 
the tangent space of $\cH^T$ at $\phi$ is naturally 
identified with the smooth functions space 
$C^\infty(X: \R)^T$ which equips the canonical $L^p$ 
metric 
\begin{equation}
 \norm{u}_p := \bigg[ \frac{1}{V}\int_X \abs{u}^p \omega_\phi^n \bigg]^\frac{1}{p}, 
\end{equation} 
where $V=\int_X \omega_\phi^n$ is the volume. 
Recent development of variational approach 
to the Monge-Amp\`ere equation 
reveals that the topology of $\cH^T$ 
induced by the $L^1$-metric 
is rather important. 
(See \cite{BBEGZ19}, \cite{BHJ17} 
for this point.) 
We denote by $d_1$ 
the distance function 
of the above $L^1$-metric structure.
It is a bit confusing but 
note that the $d_1$-convergence 
is much stronger than the 
ordinal convergence $\int_X \abs{\phi_j -\phi}\omega_0^n \to 0$ 
of $L^1$ functions. 
The $d_1$-topology is closely related to the Monge-Amp\`ere energy 
\begin{equation}
E(\phi)=\frac{1}{(n+1)V}\sum_{i=0}^{n}
\int_{X}\phi\omega_{0}^{i}\wedge\omega_{\phi}^{n-i}, 
\end{equation}
which is pluripotential generalization of 
the Dirichlet energy and indeed 
characterized by the exterior derivative property
\begin{equation}\label{differential formula}
    d_\phi E = V^{-1}\omega_\phi^n. 
\end{equation} 
In fact 
the energy $E(\phi) \in [-\infty, \infty)$ is defined for an arbitrary 
$\omega_0$-plurisubharmonic (psh in short) function $\phi$. 
Extending the classical idea of \cite{BT76}, 
in \cite{BEGZ10} the authors also extended the definition of 
the Monge-Amp\`ere product $V^{-1}\omega_\phi^n$ to 
general $\omega_0$-psh functions, as a measure 
$\MA(\phi)$ 
which contains no mass on any pluripolar set. 
The differential formula (\ref{differential formula}) 
then still holds in the appropriate sense. 
See the textbook \cite{GZ17} for the exposition. 
According to the seminal work of Darvas \cite{Dar17, Dar15}, 
the space of finite energy $\omega_0$-psh functions 
\begin{equation}
    \cE^1(X, \omega_0)^T:=
\Set{\phi\in\mathrm{PSH}(X,\omega_{0})^T :
\int_{X}\omega_{\phi}^{n}=V, \quad \int_{X}\phi\omega_{\phi}^{n}>-\infty}
\end{equation}
precisely gives the completion 
of $\cH(X, \omega_0)^T$ 
endowed with the distance $d_1$. 

According to the theory of \cite{BW14}, 
one can also define $V^{-1}g_{\phi}\omega_{\phi}^{n}$ 
as a non-pluripolar probability measur $\MA_g(\phi)$, for general $\omega_0$-psh $\phi$ 
and for a fixed positive function $g \colon P \to \R$ which was discussed in the previous subsection. 
At the same time they defined the $g$-twisted 
Monge-Amp\`ere energy which enjoys the property 
\begin{equation}
d_\phi E_{g} =V^{-1} g_\phi \omega_\phi^n, 
\end{equation}
at least for smooh $\phi$. 
If one sets $g \equiv 1$ it recovers 
the original Monge-Amp\`ere energy. 
It is then straightforwad to extend the definition of the classical energy $I$ and $J$ of 
Aubin to this setting. 
For a smooth $\phi$, 
they are written down as 
\begin{align}
    &I_{g}(\phi)=\frac{1}{V}\int_{X}\phi(g_{0}\omega_{0}^{n}-g_{\phi}\omega_{\phi}^{n}), \\
    &J_{g}(\phi)=\frac{1}{V}\int_{X}\phi g_{0}\omega_{0}^{n}-E_{g}(\phi). 
\end{align}
A simple integration by parts yields 
$0\leq I\leq (n+1)(I-J)\leq n I$. 
According to (\ref{bound of g}), 
one has the uniform estimate $\frac{1}{C}I\leq I_{g} \leq CI$, $\frac{1}{C}J\leq J_{g} \leq CJ$
and $\frac{1}{C}(I-J)\leq(I_{g}-J_{g})\leq C(I-J)$. 
The reader can consult with \cite{HaLi20} for the detail 
computations. 

Now we introduce the canonical energy functionals 
$D$ and $M$ which are originally introduced to study the 
K\"ahler-Einstein and the constant scalar curvature K\"ahler metric. 
The principal term of each energy is 
the $L$-functional and the relative entropy 
(of the two probability measure $\mu, \mu_0$), 
defined respectively as 
\begin{equation}\label{L-functional}
L(\phi)=-\log\frac{1}{V}\int_{V}e^{-\phi+\rho_{0}}\omega_{0}^{n}
\quad\text{and}\quad
\Ent(\mu \vert \mu_0)=
\int_{X}\log\Big(\frac{\mu}{\mu_0}\Big)\mu
.\end{equation}
The D-energy and the Mabuchi's K-energy 
are defined as 
\begin{align}
&D(\phi):=L(\phi)-E(\phi) \quad\text{and} \\
&M(\phi):=\Ent(\MA(\phi) \vert \MA(0))+\frac{1}{V}\int_{X}\phi\omega_{\phi}^{n}-E(\phi)
+\frac{1}{V}\int_{X}\rho_{0}(\omega_{0}^{n}-\omega_{\phi}^{n}). 
\end{align}
It is not so hard to check 
that the critical points of each energy give 
the K\"ahler-Einstein and cscK metrics, respectively. 
According to \cite{Berm13b}, these two energies can be related 
from the thermodynamical point of view. 
The reader can also consult with \cite{BBEGZ19}. 
The Legendre duality between the functions and measures 
introduces the pluricomplex energy of 
the probability measure $\mu$ as 
\begin{equation}
 E^*(\mu) := \sup_{u \in \cE(X, \omega_0)} \bigg[ E(u) -\int_X u d\mu  \bigg]
  \quad \in (-\infty, \infty]. 
\end{equation}
Then the Helmholtz free energy of the given 
probability measure $\mu$ is defined to be 
\begin{equation}
 F(\mu): = \Ent(\mu \vert \mu_0) -E^*(\mu). 
\end{equation}
In this context $D$ and $M$ 
appear in the form 
\begin{align}
&\Ent(\mu \vert \mu_0) 
=\sup_{f \in C^0(X; \R)} \bigg[ \int_X f d\mu -\log \int_X e^{f}d\mu_0 \bigg], \\ 
&M(\phi)=F(\MA(\phi)), 
\end{align}
where the background measure 
is chosen to be $\mu_0=\MA(0)$. 
The above relation is essential in showing 
the equivalence of the coercivity of two energies $D$ and $M$. 

Let us discuss the 
$g$-twisted version of the canonical energies, 
which are defined as 
\begin{equation}\label{def of D_g and M_g}
D_{g}=D+E-E_{g}
\quad\text{and}\quad
M_{g}=M+E-E_{g}. 
\end{equation}
Critical points of $D_{g}$ and $M_{g}$ 
gives $g$-solitons and $g$-extremal metric, respectively.

\begin{rem}\label{F_g and M_g}
One can also define the $g$-twisted free energy 
$F_g(\phi):=F(\MA_g(\phi))$ which is explicitly 
written down to the form 
\begin{equation}
F_g(\phi):= 
\Ent(\MA_g(\phi)\vert \MA_g(0))
+\frac{1}{V}\int_{X}\phi g_{\phi}\omega_{\phi}^{n}-E_{g}(\phi)
+\frac{1}{V}\int_{X}\rho_{0}(g_{0}\omega_{0}^{n}-g_{\phi}\omega_{\phi}^{n}). 
\end{equation}
The functional plays an important role 
in the study of $g$-solitons, 
however, the critical point of $F_g$
does not give the extremal metrics 
but $g$-solitons. 
See \cite{HaLi20} for the detail. 
The fact $F_g \neq M_g$ for general $g$
causes 
difficulty in comparing  
$D_g$ with $M_g$. 
This is why we need to develop 
another approach in showing 
$(3) \Rightarrow (1)$ of Theorem \ref{main cor}. 
\end{rem}

We here prepare for the 
later use an explicit difference 
between $M_g$ and $D_g$. 

\begin{lem}\label{M-D}
For any $\phi\in\cH^T$,
$$M_{g}(\phi)-D_{g}(\phi)
=-\frac{1}{V}\int_{X}\rho_{\phi}\omega_{\phi}^{n}
+\frac{1}{V}\int_{X}\rho_{0}\omega_{0}^{n}.$$
In particular, $M_{g}(\phi)-D_{g}(\phi)\geq \frac{1}{V}\int_{X}\rho_{0}\omega_{0}^{n}$.
\end{lem}
\begin{proof}
By the definition of $D_{g}$ and $M_{g}$, it suffice to show the equality for $g\equiv 1$.
The equality for $g\equiv 1$ was firstly observed by Ding-Tian \cite{DT92}.
The inequality follows from the fact that $e^{x}\geq 1+x$ and the normalization of the 
Ricci potential.
\end{proof}

Nextly we introduce 
the coercivity property of the 
energy, which corresponds to 
algebraic (so-called uniform) stability and ensures the existence 
of a critical point. 
See \cite{Hi20} for the validity 
of the following definition 
which is modified by a (possibly non-maximal) torus. 

\begin{dfn}
We say that a functional $F\colon \cH^T \to\mathbb{R}$ is $T_\C$-coercive if 
there exist uniform constants $\d,C>0$ such that for any $\phi\in\cH^T$ 
\begin{equation}
F(\phi)\geq \d \inf_{\sigma\in T_\C} J(\sigma[\phi])-C 
\end{equation}
holds, where $\sigma[\phi]$ is defined by 
$\sigma^{*}\omega_{\phi}=\omega_{0}+\dd\sigma[\phi]$.
\end{dfn}
Note that we can also use the functional $I_{g}-J_{g}$ to define the coercivity,
since $J$ and $I_{g}-J_{g}$ are equivalent, 
as it was explained before. 

\begin{thm}\label{coerciveD}
A Fano manifold $X$ admits a $g$-soliton in $\cH^T$ if and only if
 $D_g$ is $T_\C$-coercive.
It admits a $g$-extremal metric in $\cH^T$ if and only if
 $M_g$ is $T_\C$-coercive.
\end{thm}
\begin{proof}
For the $g$-soliton, 
this is due to \cite{HaLi20}, Theorem $3.5$. See also \cite{NN24}. 
For the $g$-extremal metric, we can adapt the argument of \cite{NN24}
which generalizes He's extension \cite{He19} of Chen-Chen's result \cite{CC1}.
We also refer the proof of \cite{ALN24} Theorem $5$.
Indeed more general weighted constant scalar curvature K\"ahler metric case 
is treated in \cite{AJL23, HaLiu24}.
\end{proof}

In view of the inequality in Lemma \ref{M-D} and above theorems,
we can conclude that if a Fano manifold admits a $g$-soliton 
it also admits a $g$-extremal metric. 
This proves the direction $(1) \Rightarrow (3)$ in Theorem \ref{main cor}.

\section{Continuity method for $g$-solitons}

\subsection{$g$-twisted Laplacian}
We starts 
from discussing 
the linearized equation 
for the continuity method. 
Let $\Delta_{\phi}:=-\dbar^{*}\dbar$ be the negative Laplacian
with respect to the K\"ahler metric $\phi\in\cH^T$, 
acting on smooth functions space $C^\infty(X; \R)$. 
The adjoint operator $\dbar^{*}$ is taken 
with respect to the natural Hermitian inner product
\begin{equation}
    \ip{u, v}_{\phi}=\frac{1}{V}\int_{X}u\overline{v}\omega_{\phi}^{n}. 
\end{equation}
Define the $g$-twisted Laplacian for $\phi\in\cH^T$ acting on functions as
\begin{equation}
\Delta_{g,\phi}f=-g_{\phi}^{-1}\dbar^{*}(g_{\phi}\dbar f). 
\end{equation} 
Note that $\Delta_{g,\phi}$ is the one-half of the weighted Laplacian
used by Di Nezza-Jubert-Lahdili \cite{DJL24} 
and Boucksom-Jonsson-Trusiani \cite{BJT24}.

\begin{lem}\label{gLap}
We have the following properties for $\Delta_{g,\phi}$.
\begin{itemize}
\item[$(1)$] The operator $\Delta_{g,\phi}$ is elliptic and 
self-adjoint with respect to
the $g$-twisted Hermitian inner product
\begin{equation}
\ip{u, v}_{g, \phi}:=\frac{1}{V}\int_{X}u\overline{v}
g_{\phi}\omega_{\phi}^{n}. 
\end{equation}
\item[$(2)$] The kernel of $\Delta_{g,\phi}$ consists of the constant functions.
\item[$(3)$] Let $g_{i\bar{j}}$ be the metric tensor of $\omega_{\phi}$. Then
$$\Delta_{g,\phi}f=\Delta_{\phi}f
+\langle\dbar\log g_{\phi}, \dbar f\rangle
=\Delta_{\phi}f+g^{i\bar{j}}\partial_{\bar{j}}f\partial_{i}\log g_{\phi}.$$ 
\item[$(4)$] For any $\phi\in\cH^T$, let $\phi_{t}$ be a path in
$\cH^T$ with $\phi_{0}=\phi$ and $\frac{d}{dt}\big\vert_{t=0}\phi_{t}=:u$.
Then 
\begin{equation}
\frac{d}{dt}\bigg\vert_{t=0}(g_{\phi_{t}}\omega_{\phi_{t}}^{n})=
(\Delta_{g,\phi}u)g_{\phi}\omega_{\phi}^{n}. 
\end{equation}
\end{itemize}
\end{lem}
\begin{proof}
The claims $(1)$, $(2)$ and $(3)$ follow from the definition of $\Delta_{g,\phi}$ immediately.   
We give a proof of $(4)$.
Recall that $\xi_{1},\dots,\xi_{r}$ are 
the generating vector fields for 
$T_\C \simeq (\mathbb{C}^{*})^{r}$. 
Let $(\theta_{1},\dots,\theta_{r})$ be the standard coordinate of
the moment polytope $P$.
In view of $(3)$, it suffice to show
\begin{equation}
\frac{d}{dt}\bigg\vert_{t=0}g_{\phi_{t}}=\langle\dbar g_{\phi},\dbar u\rangle.
\end{equation}
By simple calculations, we get
$$\frac{d}{dt}\bigg\vert_{t=0}g_{\phi_{t}}
=\sum_{\alpha=1}^{r}\frac{\partial g_{\phi}}{\partial\theta_{\alpha}}\xi_{\alpha}(u)
\quad\text{and}\quad
\langle\dbar g_{\phi},\dbar u\rangle
=\sum_{\alpha=1}^{r}\frac{\partial g_{\phi}}{\partial\theta_{\alpha}}
\overline{\xi_{\alpha}}(u).$$
Since any metric in $\cH^T$ is invariant under the action of the imaginary
part of $\xi_{\alpha}$, we have $\xi_{\alpha}(u)=\overline{\xi_{\alpha}}(u)$.
This completes the proof.
\end{proof}
\subsection{Continuity method}\label{conti-method}
In this subsection we prove the set 
$\cT$ 
in Theorem \ref{Thm-T} 
is non-empty and open. 
Recall we consider the continuity method \eqref{conti-eq} to construct a $g$-soliton.
Let $\mathcal{T}$ be the set of $t\in[0,1]$ such that the equation \eqref{conti-eq}
has a solution in $\cH^T$.
For this purpose, we consider the following modified equation
\begin{equation}\label{conti-eq'}
g_{\phi}\omega_{\phi}^{n}=e^{-t\phi-E_{g}(\phi)+\rho_0}\omega_0^n 
\quad\text{for}\quad t\in [0,1].
\end{equation}
Let $\mathcal{T}^{*}$ be the set of $t\in[0,1]$ such that the equation \eqref{conti-eq'}
has a solution in $\cH^T$.
For any $t\in(0,1]$ there is a one-to-one correspondence between the solution of \eqref{conti-eq} 
and the solution of \eqref{conti-eq'}. 
The linearized operator associated from \eqref{conti-eq}, 
however, is not invertible 
at $t=0$ as discussed below.

\subsubsection{Solution at $t=0$}\label{sol-at-0}
According to \cite{BW14}, Theorem 1.2, there exists a $T$-invariant continuous solution
$\phi$ of $g_{\phi}\omega_{\phi}^{n}=e^{\rho_{0}}\omega_{0}$.
The regularity argument \cite{HaLi20}, Proposition 3.8
(by using the log concavity assumption for $g$)
shows $\phi$ is smooth.
One can also apply the argument of \cite{DJL24, HaLiu24}.
Up to addition of a constant $\phi$ satisfies \eqref{conti-eq'} at $t=0$.
Therefore $0\in\mathcal{T}^{*}$.

\subsubsection{The openness of $\mathcal{T}^{*}$ at $t=0$}
For any $\phi\in \cH^T$, 
define $$\Phi^{*}(\phi)=\log\Big(\frac{g_{\phi}\omega_{\phi}^{n}}{\omega_{0}^{n}}\Big)+E_{g}(\phi)-\rho_{0}.$$
By lemma \ref{gLap}, the linearized operator at $t=0$ is
$\delta\Phi^{*}(u)=\Delta_{g,\phi}u+V^{-1}\int_{X}ug_{\phi}\omega_{\phi}^{n}$, 
where $u$ is any variation at a solution $\phi$ of \eqref{conti-eq'} at $t=0$.
Since the kernel of $\delta\Phi^{*}$ is trivial, the implicit function theorem shows
the openness of $\mathcal{T}^{*}$ at $t=0$. 

\subsubsection{The openness of $\mathcal{T}\cap(0,1)$}
For $\phi\in\cH^T$, define 
$$\Phi(\phi)
=\log\Big(\frac{g_{\phi}\omega_{\phi}^{n}}{\omega_{0}^{n}}\Big)+t\phi-\rho_{0}.$$
By Lemma \ref{gLap}, the linearized operator is $\delta\Phi(u)=\Delta_{g,\phi}u+tu$, 
where $u$ is any variation at a solution $\phi$ of \eqref{conti-eq} at $t$.
The openness of $\mathcal{T}\cap(0,1)$ follows from the next proposition 
combined with the implicit function theorem.

\begin{prop}\label{eigenvl}
Let $\phi_{t}\in\cH^T$ be the solution of \eqref{conti-eq} at $t\in(0,1)$.
The first eigenvalue of the operator $\Delta_{g,\phi_{t}}+t$ is 
 negative.
\end{prop}

\begin{proof}
Let $g_{i\bar{j}}$ be the metric tensor for $\omega_{\phi_{t}}$ 
and $\lambda$ be the first non-zero eigenvalue of $\Delta_{g,\phi_{t}}$, 
and $u\in C^{\infty}(X,\mathbb{R})$ be the eigenvector.
Put $f=\log g_{\phi_{t}}$ for simplicity.
Applying $\nabla_{\bar{k}}$ to the equation $\Delta_{g,\phi_{t}}u=\lambda u$
we have
\begin{align}
\lambda\nabla_{\bar{k}}u
&=
g^{i\bar{j}}\nabla_{i}\nabla_{\bar{j}}\nabla_{\bar{k}}u
-R_{\bar{k}}^{\bar{p}}\nabla_{\bar{p}}u
+g^{i\bar{j}}\nabla_{\bar{j}}u\nabla_{\bar{k}}\nabla_{i}f
+g^{i\bar{j}}\nabla_{\bar{k}}\nabla_{\bar{j}}f\nabla_{i}f \\
&<
g^{i\bar{j}}\nabla_{i}\nabla_{\bar{j}}\nabla_{\bar{k}}u
-t\nabla_{\bar{k}}u
+g^{i\bar{j}}\nabla_{\bar{k}}\nabla_{\bar{j}}u\nabla_{i}f \\
&=e^{-f}\nabla_{i}\Big(e^{f}\nabla_{\bar{j}}\nabla_{\bar{k}}u\Big)
-t\nabla_{\bar{k}}u.
\end{align}
In the above inequality we used the equation \eqref{conti-eq} 
in the form 
\begin{equation}
\Ric(\omega_{\phi_{t}})-\dd\log g_{\phi_{t}}=(1-t)\omega_{0}+t\omega_{\phi_{t}}
>t\omega_{\phi_{t}}. 
\end{equation}
If one multiplies $g^{m\bar{l}}\nabla_{m}ug_{\phi_{t}}\omega_{\phi_{t}}$ 
to the both sides 
then a simple integration by part shows 
$$\lambda\int_{X}|\bar{\nabla}u|^{2}g_{\phi_{t}}\omega_{\phi_{t}}^{n}
<-\int_{X}|\bar{\nabla}\bar{\nabla}u|^{2}g_{\phi_{t}}\omega_{\phi_{t}}^{n}
-t\int_{X}|\bar{\nabla}u|^{2}g_{\phi_{t}}\omega_{\phi_{t}}^{n}.$$
Therefore $\lambda+t<0$.
\end{proof}

Combining the above results together, 
we conclude that $\cT$ is non-empty and open. 

\section{Estimate of $M_{g}$ along the continuity method}

In this section we prove Theorem \ref{Thm-M}.
Let $\phi_{t}\in\cH^T$ be a solution of the equation \eqref{conti-eq} at $t$.
In order to obtain the upper bound of $M_{g}(\phi_{t})$, 
it suffices to control $D_{g}(\phi_{t}) -\int_{X}\rho_{\phi}\omega_{\phi}^{n}$, 
in view of Lemma \ref{M-D}.
We put $\omega_{t}:=\omega_{\phi_{t}}, \rho_{t}:=\rho_{\phi_{t}}$, $g_{t}:=g_{\phi_{t}}$
and $\Delta_{t}:=\Delta_{\phi_{t}}$ for simplicity.

\begin{lem}\label{rho-eq}
Let $\phi_{t}\in \cH^T$ be a solution of \eqref{conti-eq} at $t$. 
One has the relation 
\begin{equation}
\rho_{t}+(1-t)\phi_{t}-\log g_{t}
-L(\phi_t)=0. 
\end{equation} 
\end{lem}
\begin{proof}
From the equation \eqref{conti-eq} and definition of the Ricci potential, 
we observe 
\begin{align}
0&=-\dd\log g_{\phi_{t}}+\Ric(\omega_{t})-\Ric(\omega_{0})
+\dd(\rho_{0}-t\phi_{t}) \\
&=\dd\Big(\rho_{t}+(1-t)\phi_{t}-\log g_{\phi_{t}}\Big).
\end{align}
Thus $\rho_{t}+(1-t)\phi_{t}-\log g_{\phi_{t}}=c_{t}$ for some constant $c_{t}$
depending on $t$.
The normalization of $\rho_{t}$ and the equation \eqref{conti-eq} forces 
\begin{equation}
c_{t}=-\log\frac{1}{V}\int_{X}e^{-(1-t)\phi_{t}}g_{\phi_{t}}\omega_{t}^{n}
=-\log\frac{1}{V}\int_{X}e^{-\phi_t+\rho_{0}}\omega_{0}^{n}. 
\end{equation} 
The right-hand side is the $L$-functional 
introduced by (\ref{L-functional}). 
\end{proof}
Integrating the both sides of Lemma \ref{rho-eq},
one obtains 
\begin{equation}\label{est-ric}
-\frac{1}{V}\int_{X}\rho_{t}\omega_{t}^{n}
\leq \frac{1-t}{V}\int_{X}\phi_{t}\omega_{t}^{n} -\log(\inf_X g)
-L(\phi_t). 
\end{equation}
Thanks to the assumption $g>0$, the first term is bounded from above, as follows. 

\begin{lem}\label{int-phi}
For a fixed $\e\in (0,1)$ 
the solution $\phi_{t}$ at $t>\e$ satisfies
\begin{equation}
\frac{1}{V}\int_{X}\phi_{t}\omega_{t}^{n}\leq\frac{e^{\sup\rho_{0}-1}}{\inf_X g}\e^{-1}.
\end{equation}
\end{lem}

\begin{proof} 
Using the equation \eqref{conti-eq} of the continuity method, 
we transrate the volume form into the form 
\begin{eqnarray*}
\frac{1}{V}\int_{X}\phi_{t}\omega_{t}^{n}
&\leq&
\frac{1}{V}\int_{\{\phi_{t}>0\}}\phi_{t}g_t^{-1}e^{-t\phi_{t}+\rho_{0}}
\omega_{0}^{n}\\
&\leq&
\frac{e^{\sup\rho_{0}}}{V\inf_X g}
\int_{\{\phi_{t}>0\}}\phi_{t}e^{-t\phi_{t}}\omega_{0}^{n}.
\end{eqnarray*}
This completes the proof, since for any $t>\e$ the function $\mathbb{R}_{>0}\ni x\mapsto xe^{-tx}$
is bounded from above by the constant $(e\e)^{-1}$. 
\end{proof}

Although the third term $L(\phi_t)$ in the right hand side in \eqref{est-ric}
is not bounded from above,
this is cancelled by the same term in $D_{g}$. 
The remaining part is in fact  
uniformly bounded from above 
regardless of $m_X$ or $\e$. 

\begin{lem}\label{I-J is non-decreasing}
    The functional $I_{g}-J_{g}$ is 
    non-decreasing along the continuity path
    \eqref{conti-eq}. 
\end{lem}
\begin{proof}
 The same result for $I-J$ is proved in \cite{BM87} page $28$. 
 If one differentiates \eqref{conti-eq}, 
 Lemma \ref{gLap} implies 
 \begin{equation}
  \Delta_{g, t} \dot{\phi}_t +t\dot{\phi}_t  +\phi_t=0. 
 \end{equation}
 So we may compute as 
\begin{align}
\frac{d}{dt} (I_g-J_g) (\phi_t) 
&= -\frac{1}{V}\int_X \phi_t \Delta_{g, t} \dot{\phi_t} g_t  \omega_t^n\\
&=  \frac{1}{V}\int_X (\Delta_{g, t}\dot{\phi}_t +t\dot{\phi}_t)\Delta_{g, t} \dot{\phi_t} g_t  \omega_t^n \\
&= \frac{1}{V}\int_X (\Delta_{g, t}\dot{\phi}_t + t\dot{\phi}_t)^2  \omega_t^n 
-\frac{t}{V}\int_X \dot{\phi}_t (\Delta_{g, t}\dot{\phi}_t + t\dot{\phi}_t)g_t  \omega_t^n. 
\end{align}
The second term in the last line 
is non-negative by Proposition \ref{eigenvl}. 
\end{proof}

\begin{prop}
Along the solution $\phi_{t}$, 
the $g$-twisted Monge-Amp\`ere energy is non-negative: $E_g(\phi_t) \geq 0$. 
\end{prop}

\begin{proof}
We first show that
one can extend the solution $\phi_{t}$ at $t$ to a family of solutions
$\phi_{s}$ for $s\in[0,t]$. 
Let $\mathcal{T}_{t}$ be the set of $s\in[0,t]$ such that the equation
\eqref{conti-eq} has a solution in $\mathcal{H}^{T}$.
We already showed in Section \ref{conti-method}
that $\mathcal{T}_{t}$ is non-empty and open.
By Lemma \ref{I-J is non-decreasing},
$(I_{g}-J_{g})(\phi_{s})\leq(I_{g}-J_{g})(\phi_{t})$ for any $s\in\mathcal{T}_{t}$.
Thus we may apply the argument of Proposition \ref{t-infty}
to show that $\mathcal{T}_{t}$ is closed.

In order to show $E_g(\phi_t) \geq 0$,
it is sufficient to prove the formula 
\begin{equation}\label{Ev}
E_{g}(\phi_{t})=\frac{1}{t}\int_{0}^{t}\Big(I_{g}(\phi_{s})-J_{g}(\phi_{s})\Big)ds, 
\end{equation}
because $I_{g}-J_{g}$ is non-negative.
This is also well-known to the experts 
but we give a proof of \eqref{Ev} 
for the convenience to the reader. 
From the general expression 
\begin{equation}\label{expression of I-J}
I_{g}(\phi)-J_{g}(\phi)=E_{g}(\phi)-\frac{1}{V}\int_{X}\phi g_{\phi}\omega_{\phi}^{n},
\end{equation}
it is reduced to the differential formula 
\begin{equation}\label{der-path}
t\frac{d}{dt}E_{g}(\phi_{t})=-\frac{1}{V}\int_{X}\phi_{t}g_{t}\omega_{t}^{n}.
\end{equation}
Let us differentiate the Monge-Amp\`ere energy 
along the continuty path. 
We denote $\dot{\rho_{t}}:=\frac{d}{dt}\rho_{t}$ and $\dot{\phi_{t}}:=\frac{d}{dt}\phi_{t}$ 
for simplicity. 
If one differentiates the identity of Lemma \ref{rho-eq} first, 
he observes 
\begin{equation}\label{der-rho-eq}
\dot{\rho_{t}}=\phi_{t}-(1-t)\dot{\phi_{t}}
-\langle \dbar\log g_{t}, \dbar\dot{\phi_{t}}\rangle
+\frac{1}{V}\int_X \dot{\phi}_t e^{\rho_t}\omega_t^n.  
\end{equation}
On the other hand, the definition of the Ricci potential 
yields 
\begin{equation}\label{der-rho}
\dot{\rho_{t}}=\Delta_{t}\dot{\phi_{t}}-\dot{\phi_{t}}+
\frac{1}{V}\int_{X}\dot{\phi_{t}}e^{\rho_{t}}\omega_{t}^{n}, 
\end{equation}
where $\Delta_{t}$ is the geometric Laplacian with respect to $\omega_t$. 
Using the two equations \eqref{der-rho-eq} and \eqref{der-rho}, 
one can compute as
\begin{align}
\frac{d}{dt}E_{g}(\phi_{t})
&=\frac{1}{V}\int_{X}\dot{\phi_{t}}g_{t}\omega_{t}^{n}\\
&=\frac{1}{V}\int_{X}\Big( -\phi_{t}+(1-t)\dot{\phi_{t}}
+\langle \dbar\log g_{t}, \dbar\dot{\phi_{t}}\rangle
+\Delta_{t}\dot{\phi_{t}}\Big)
g_{t}\omega_{t}^{n}\\
&=-\frac{1}{V}\int_{X}\phi_{t}g_{t}\omega_{t}^{n}
+(1-t)\frac{d}{dt}E_{g}(\phi_{t}). 
\end{align}
In the last line we used Lemma \ref{gLap}, $(3), (4)$ which tells 
\begin{equation}
0=\frac{d}{dt}\int_{X}g_{t}\omega_{t}^{n}
=\int_{X}\Big(
\Delta_{t}\dot{\phi_{t}}+
\langle \dbar\log g_{t}, \dbar\dot{\phi_{t}}\rangle
\Big)
g_{t}\omega_{t}^{n}. 
\end{equation} 
It shows \eqref{der-path}. 
\end{proof}

Summarizing the above argument, we obtain the following:
Fix any $\e\in(0,1)$. 
Let $\phi_{t}\in\cH^T$ be a solution of the equation \eqref{conti-eq} at $t>\e$.
Then one has the estimate 
$$M_{g}(\phi_{t})\leq
\frac{e^{\sup\rho_{0}-1}}{\inf g}\e^{-1}
-\log(\inf_X g)+
\frac{1}{V}\int_{X}\rho_{0}\omega_{0}^{n}.$$
This completes the proof of Theorem \ref{Thm-M}.

\section{The closedness of $\mathcal{T}$}\label{T-closed}
In this last section we complete the proof of Theorem \ref{Thm-T}.
It remains to show the closedness of
$\mathcal{T}$, when there exists an extremal K\"ahler metric. 
Thanks to Theorem \ref{coerciveD}, 
we may assume the coercivity of $M_{g}$. 
As in the previous section, we will write as 
$\omega_{t}:=\omega_{\phi_{t}}, \rho_{t}:=\rho_{\phi_{t}}$ and $g_{t}:=g_{\phi_{t}}$.

\subsection{Control of $I_g-J_g$} 

If the coercivity of $M_{g}$ is assumed,
then $\inf_{\sigma\in T_{\mathbb{C}}}(I_{g}-J_{g})(\sigma[\phi_{t}])$
is bounded by Theorem \ref{Thm-M}.
The first step is to show 
that one can in fact control 
$(I_g-J_g)(\phi_{t})$ along the continuity path $\phi_{t}$. 
It extends \cite{LZ}, Lemma $3.3$ 
to the general $g$ case.

The following lemma rephrases 
the equivalence of the two expressions 
for the (modified) Fuktaki invariant, 
in terms of the energies. 
\begin{lem}\label{LemdMD}
Let $\xi \in \ft_\C$ and $\{\sigma_{s}\}_{s\in\mathbb{R}}$ 
the one-parameter subgroup generated by
the real part $\Re(\xi)$. 
Fix any $\phi\in\cH^T$. Then,
\begin{equation}\label{dMD}
\frac{d}{ds}M_{g}(\sigma_{s}[\phi])=\frac{d}{ds}D_{g}(\sigma_{s}[\phi]).
\end{equation}
If the functional $M_{g}$ is bounded from below on $\cH^T$, then 
the both sides vanish identically. 
\end{lem}
\begin{proof}
Since $\sigma_{s}^{*}\omega_{\phi}=\omega_{\sigma_{s}[\phi]}$ 
and $\sigma_{s}^{*}\rho_{\phi}=\rho_{\sigma_{s}[\phi]}$, 
\eqref{dMD} follows from Lemma \ref{M-D}.
We put $\psi:=\sigma_{s}[\phi]$ and starts from the expression 
\begin{equation}
\frac{d}{ds}D_{g}(\sigma_{s}[\phi])
=\frac{1}{V}\int_{X}\mathrm{Re}(\theta_{\xi}(\psi))
(e^{\rho_{\psi}}-g_{\psi})\omega_{\psi}^{n},
\end{equation}
where $\theta_{\xi}(\psi)$ is the Hamilton function. 
If we differentiates the defining equation of the Ricci potential 
along $\xi$, we observe
\begin{equation}
\Delta_{\psi}\theta_{\xi}(\psi)+\theta_{\xi}(\psi)+\xi \rho_{\psi} 
=\frac{1}{V}\int_{X}\theta_{\xi}(\psi)e^{\rho_{\psi}}\omega_{\psi}^{n} 
\end{equation}
(see \eg \cite{NN22}, Lemma 3.1). 
Therefore we derive 
\begin{equation}
    \frac{d}{ds}D_{g}(\sigma_{s}[\phi])
    =\frac{1}{V}\int_{X}\mathrm{Re}(\xi)(\rho_{\psi}-\log g_{\psi})
    g_{\psi}\omega_{\psi}^{n}. 
\end{equation}
The right-hand side is the real part of the $g$-twisted Futaki invariant 
(see \cite{HaLi20}, Definition $4.1$). 
In particular, it does not depend on the  choice of $\psi$ 
and $s$. 
If $M_{g}$ is bounded from below, 
therefore, it identically vanishes. 
\end{proof}

\begin{prop}\label{I-J}
Assume the functional $M_{g}$ is bounded from below on $\cH^T$.
Let $\phi_{t}\in\cH^T$ be a solution of \eqref{conti-eq} at $t\in[0,1)$.
Then we have
$$\inf_{\sigma\in T_\C}(I_{g}-J_{g})(\sigma[\phi_{t}])=(I_{g}-J_{g})(\phi_{t}).$$
\end{prop}
\begin{proof}
In the same notation as the above lemma, 
a direct computation using (\ref{expression of I-J}) 
shows 
\begin{align}
\frac{d}{ds}\bigg\vert_{s=0}(I_{g}-J_{g})(\sigma_{s}[\phi_{t}])
&=-\frac{1}{V}\int_{X}\frac{d}{ds}\bigg\vert_{s=0}\sigma_{s}[\phi_{t}]
(\Delta_{g,\phi_{t}}\phi_{t})g_{t}\omega_{t}^{n} \\
&= \frac{1}{V}\int_{X}\ip{\dbar\Re(\theta_\xi(\phi_{t})), \dbar \phi_{t}}
g_{t}\omega_{t}^{n}. 
\end{align}
If one takes notice of the definition of the Hamilton function and Lemma \ref{rho-eq}, 
it is equivalent to 
\begin{align} 
-\frac{1}{(1-t)V}\int_{X}\Re(\xi)(\rho_{t}-\log g_{t})
g_{t}\omega_{t}^{n}
= -\frac{1}{1-t} \frac{d}{ds}\bigg\vert_{s=0}D_{g}(\sigma_{s}[\phi]). 
\end{align}
According to Lemma \ref{LemdMD}, 
it implies that 
$I_g-J_g$ is critical at $\s=\id$. 

We will show that for any fixed $\phi\in\cH^T$, the function 
$s\mapsto f(s):=(I_{g}-J_{g})(\sigma_{s}[\phi])$ is properly convex 
(so that the infimum is achieved by $\s=\id$). 
Similar argument for the $J$, $D$, and $M$ 
has appeared in \cite{Hi19-2}, Theorem $2.5$. 
See also \cite{HaLi20}, Lemma $3.5$ and $4.6$ 
for the detail of the fiber integration below. 
We consider the complex variable $s\in\mathbb{C}$ 
and define $\Omega(z,s)=\sigma_{s}^{*}(\omega_{\phi}(z))$ as a 
smooth semipositive form on the product space $X\times\mathbb{C}$. 
If regards the pulled-back form $p_1^*\omega_0$  by the first projection 
as a reference form, one has a potential function $\Psi(z,s)$ 
such that 
$\Omega=p_1^*\omega_{0}+\dd\Psi$. 
In terms of the fiber integration 
the second variation of the Monge-Amp\`ere 
energy (Recall that the first variation 
was given by (\ref{differential formula})) 
is expressed as 
\begin{equation}
\dd E_{g}(\Psi)=\frac{1}{V}\int_{X}g_{\Psi}\Omega^{n+1}, 
\end{equation}
where the operator $\partial$ is taken for $n+1$ variables 
$(z, s)$. 
It actually vanishes since $\Omega^{n+1}=\sigma_{s}^{*}(\omega_{\phi}^{n+1})=0$.
On the other hand,  
the integration of $\Psi$ against the $g$-twisted 
Monge-Amp\`ere measure on the product space 
is computed as 
\begin{align}
-\dd \int_{X}\Psi g_{\Psi}\Omega^{n} 
&=-\int_{X}\dd\Psi\wedge g_{\Psi}\Omega^{n} \\
&=-\int_{X}g_{\Psi}\Omega^{n+1}
+\int_{X}p_1^*\omega_0\wedge g_{\Psi}\Omega^{n} \\
&=\int_{X}p_1^*\omega_0\wedge g_{\Psi}\Omega^{n}. 
\end{align}
The last term is obviously non-negative. 
It implies that $(I_{g}-J_{g})(\s_s [\phi])$ is subharmonic 
in $s$. 
Properness of $f(s)$ follows from the fact that the slope at infinity of $f(s)$
is positive unless $\Re(\xi) =0$.
\end{proof}
\subsection{Closedness of $\mathcal{T}$}\label{pr-closed}
Let us now prove the closedness of $\cT$ 
so that conclude Theorem \ref{Thm-T}.
\begin{prop}\label{t-infty}
Suppose that the functional $M_{g}$ is coercive on $\cH^T$.
Take $t_j \in \cT$ which converges to $t_\infty \in \R$, 
and $\phi_j\in\cH^T$ as the solution of \eqref{conti-eq} at $t=t_{i}$.
Replaced with a subsequence if necessary, 
$\phi_j$ converges to some 
$\phi_{\infty}\in\mathcal{E}^{1}(X, \omega_0)^T$ in the $d_1$-toplogy 
and the limit function $\phi_{\infty}$ is in fact a smooth solution of \eqref{conti-eq}
at $t=t_{\infty}$.
It implies $t_{\infty}\in\mathcal{T}$.
\end{prop}
\begin{proof}
In the sequel all the constants $C$ are uniform in $j$.
We write $\omega_j:=\omega_{\phi_j}$ and $g_j:=g_{\phi_j}$.

As we have already mentioned, 
the functionals $I, J$, and $I_g-J_g$ are 
interchangeable with each other.  
According to Theorem \ref{Thm-M}
and Proposition \ref{I-J}, the coercivity of $M_{g}$ 
implies the uniform bound $J(\phi_j)\leq C$. 

Uniform bound of the $I$-functional, 
combined with Lemma \ref{int-phi} yields  
$\int_{X}\phi_j\omega_{0}^{n}\leq C$. 
Then by the standard Green function argument 
we may obtain the uniform bound of $\sup_X \phi_j$. 
Let $G_{\omega_{0}}$ be the Green function for the background metric $\omega_{0}$, 
which satisfies $\inf_X G_{\omega_{0}}\geq -B$ for some positive constant $B$.
Since $\Delta_{\omega_{0}}\phi_j\geq -n$, 
the representaion formula implies 
\begin{equation}
\sup_{X}\phi_j\leq\frac{1}{V}\int_{X}\phi_{i}(y)\omega_{0}^{n}(y)
+\sup_{x\in X}\int_{X}(G_{\omega_{0}}(x,y)+B)(-2\Delta_{\omega_{0}}\phi_{i}(y))
\omega_{i}^{n}(y)\leq C.
\end{equation}
The lower bound  $\sup_{X}\phi_{i}\geq -C$ is 
rather clear since
\begin{equation}
    \int_{X}e^{-t_j\phi_j+\rho_0}\omega_{0}^{n}=\int_{X}g_j\omega_j^{n}
=\int_{X}\omega_{0}^{n}.
\end{equation}
Thus we obtain 
\begin{equation}\label{sup-bound}
\abs{\sup_{X}\phi_{i}}\leq C.
\end{equation}
The above Green function argument implies 
the general inequality 
$\sup_X \phi \leq V^{-1}\int_X \phi \omega_0^n +C $ 
so the uniform estimate of $I(\phi_{i})$ 
now implies $|\int_{X}\phi_j\omega_j^{n}|<C$.
If look back to the expression of 
$M_g$ (which is bounded from above by Theorem \ref{Thm-M}), 
we obtain the uniform bound of the relative entropy 
\begin{equation}\label{H-bound}
\Ent(\MA(\phi_j)\vert \MA(0))\leq C. 
\end{equation}
This is a key 
to extract a convergent subsequence. 
Indeed by the fundamental compactness result 
\cite{BBEGZ19}, Theorem $2.17$ and Proposition $2.6$,   
 there exists a convergent subsequence and 
the limit 
in $(\cE^1)^T=\cE^1(X, \omega_0)^T$. 
As an abuse of notation, 
we denote the subsequence by 
the same symbol $\phi_j$. 
According to \cite{Dar15}, the convergence $\phi_j\to\phi_{\infty}$ in $d_{1}$-topology
implies $\phi_j \to\phi_{\infty}$ in $L^{1}(X, \omega_0^n)$, 
$E(\phi_j)\to E(\phi_{\infty})$, 
and vice versa. 
Finiteness of the Monge-Amp\`ere energy 
forces the singularities of 
the psh funcions very mild. 
Indeed, we can apply 
the uniform version of the Skoda 
integrability theorem (\cite{Zer01}, Corollary $3.2$) 
to the sequence so that 
\begin{equation}
    \int_X e^{-pt_j \phi_j} \omega_0^n, \quad 
\int_{X}e^{-pt_{\infty}\phi_{\infty}}\omega_{0}^{n} \leq C 
\end{equation}
holds for any $p>1$. 
The effective version of Demailly-Koll\`ar's semi-continuity theorem for
log canonical thresholds (\cite{DK01}, Main Theorem $0.2, (2)$. 
See also \cite{BBEGZ19}, Proposition $1.4$) shows
the convergence $e^{-t_j\phi_j}\to e^{-t_{\infty}\phi_{\infty}}$ in $L^{p}(\omega_{0}^{n})$
for all $p>1$.
It follows the convergence of probability measures
\begin{equation}\label{RHSmeas}
g_j\omega_j^{n}=e^{-t_j\phi_j +\rho_{0}}\omega_{0}^{n}
\to e^{-t_{\infty}\phi_{\infty}+\rho_{0}}\omega_0^n.
\end{equation}
We also claim the convergence of probability measures
\begin{equation}\label{LHSmeas}
g_j \omega_j^{n}=\MA_g(\phi_j)\to \MA_g(\phi_\infty).
\end{equation}
Note that the limit measure is a priori singular 
and $g_{\phi_\infty}$ itself does not makse sense. 
See \cite{BW14} for the detail. 
It is convenient to use the symmetric $I$-functional 
\begin{equation}
I(\psi_{1},\psi_{2}):=\int_{X}(\psi_{1}-\psi_{2})(\omega_{\psi_{1}}^{n}-\omega_{\psi_{2}}^{n}) 
\end{equation} 
which is defined for $\psi_1, \psi_2 \in \cE^1(X, \omega_0)$.  
As \cite{BBGZ13}, Lemma $3.13$, 
we observe 
\begin{eqnarray*}
\sup_{(\cE^1)^T_B}
\abs{
\int_{X}ug_{\psi_{1}}\omega_{\psi_{1}}^{n}-\int_{X}ug_{\psi_{2}}\omega_{\psi_{2}}^{n} 
}
&\leq&
C \sup_{(\cE^1)^T_B}
\abs{
\int_{X}u\omega_{\psi_{1}}^{n}-\int_{X}u\omega_{\psi_{2}}^{n} 
} \\
&\leq& C I(\psi_{1},\psi_{2})^{1/2} 
\end{eqnarray*}
for the weak compact set 
\begin{equation}
    (\cE^1)^T_B:=
    \Set{\phi\in\cE^{1}(X, \omega_0)^T: \sup_{X}\phi\leq 0, E(\phi)\geq -B}.
    \end{equation}
Following \cite{BBGZ13}, Proposition $5.6$, 
we may exploit the above inequality to 
control the ($g$-twisted) pluricomplex energy of the probability measure 
\begin{equation}
E^{*}_{g}(\mu):=\sup_{u\in (\cE^1)^T}\bigg[E_{g}(u)-\int_{X}ud\mu\bigg]. 
\end{equation}
Indeed, the convergence $\phi_{i}\to\phi_{\infty}$ in $d_{1}$-topology 
is equivalent to 
$I(\phi_j, \phi_\infty)\to 0$. 
Using the above inequality, we deduce 
$E^{*}_{g}(\MA_g(\phi_j))\to E^{*}_{g}(\MA_g(\phi_\infty))$, 
which in particular 
implies the weak convergence \eqref{LHSmeas}.

From \eqref{RHSmeas} and \eqref{LHSmeas}, we obtain 
the weak solution which satisfies 
\begin{equation}
\MA_g(\phi_\infty)=e^{-t_{\infty}\phi_{\infty}+\rho_{0}}\omega_{0}^{n}. 
\end{equation}
Since the right-hand side has $L^p$-density for some $p>1$, 
one can apply \cite{BW14}, Theorem $1.2$ 
(or more directly Ko\l odziej's a priori estimate \cite{Kol}
for the Monge-Amp\`ere equation)
so that $\phi_{\infty}$ is continuous. 
The regularity argument \cite{HaLi20}, Proposition 3.8
(by using the log concavity assumption for $g$)
shows $\phi_{\infty}$ is smooth.
One can also apply the argument of \cite{DJL24, HaLiu24}. 
At any rate it implies $t_{\infty}\in\mathcal{T}$.
\end{proof}


\end{document}